\documentclass[11pt]{amsart}
\usepackage[a4paper,margin=1.08in]{geometry}
\usepackage{amsmath,amssymb,amsthm,mathtools,mathrsfs}
\usepackage{enumitem,booktabs,array,tabularx}
\usepackage{xcolor}
\usepackage{tikz}
\usetikzlibrary{arrows.meta,calc,positioning,decorations.markings}
\usepackage[colorlinks=true,linkcolor=blue!45!black,citecolor=blue!45!black,urlcolor=blue!45!black]{hyperref}
\usepackage{microtype}

\newtheorem{theorem}{Theorem}[section]
\newtheorem{proposition}[theorem]{Proposition}
\newtheorem{corollary}[theorem]{Corollary}
\newtheorem{lemma}[theorem]{Lemma}
\newtheorem{remark}[theorem]{Remark}
\theoremstyle{definition}
\newtheorem{definition}[theorem]{Definition}

\newcommand{\HH}{\mathbb H}
\newcommand{\Fq}{\mathscr F_q}

\newcommand{\Aut}{\operatorname{Aut}}
\newcommand{\Cuts}{\mathcal C}
\newcommand{\Corners}{\mathcal K}
\newcommand{\PSL}{\operatorname{PSL}}
\newcommand{\SL}{\operatorname{SL}}

\newcommand{\Gal}{\operatorname{Gal}}

\definecolor{treeRed}{RGB}{170,40,40}
\definecolor{treeBlue}{RGB}{45,75,155}
\tikzset{
  vzero/.style={circle,fill=black,inner sep=2.2pt},
  vone/.style={circle,draw=black,fill=white,line width=.8pt,minimum size=7pt,inner sep=0pt},
  redleaf/.style={circle,draw=treeRed,fill=white,line width=1.1pt,minimum size=6.5pt,inner sep=0pt},
  blueleaf/.style={circle,draw=treeBlue,fill=treeBlue!25,line width=1pt,minimum size=6.5pt,inner sep=0pt},
  edge/.style={line width=.75pt}
}

\title{Hecke Triangle Groups and Dessins d'Enfants}
\author{Devendra Tiwari}
\address{Statistics and Mathematics Department\\
Indian Statistical Institute\\
Bangalore 560 059\\
INDIA}
\email{devendra9.dev@gmail.com}
\date{October 2026}

\begin{document}
\begin{abstract}
We study finite-index subgroups of the Hecke triangle group $H_q$ of signature $(2,q,\infty)$ through the finite-edge tessellation, special polygons, and the associated dessins.  The quotient finite-edge graph is a connected bipartite ribbon graph in which even vertices have degree $1$ or $2$ and odd vertices may have any degree $d\mid q$; geometrically, degree $d$ records a $d$-cluster and an elliptic stabilizer of order $q/d$.  We determine two finite ambiguities that arise when one passes from this quotient graph to a special polygon.  First, all tree diagrams obtained from a fixed ribbon graph are classified by the possible spanning-tree cuts modulo graph automorphisms.  Second, at a proper cluster one must record which cyclic gap contains the omitted branches of the universal $q$-star; every such choice develops to a special polygon.  Combining the two gives an explicit finite description of all polygon orbits over a fixed dessin, together with a separate count of normalized realizations in the fixed tessellation.  Concrete examples for $H_5$, $H_6$, and $H_7$ illustrate the geometry, and two arithmetic corollaries concern level-$2$ cusp labels and Galois orbits of principal congruence dessins.
\end{abstract}
\maketitle

\section{Introduction}

A finite-index subgroup of a triangle group can be described in several familiar ways: by a transitive permutation representation, by a quotient of the tessellation, by a dessin or orientable map, and, when a suitable fundamental polygon is available, by side-pairing data.  For the modular group these descriptions were organized particularly effectively through special polygons, cuboid graphs, tree diagrams, and Farey symbols \cite{Kul91}.  The aim of the present paper is to develop the corresponding finite graph--tree--polygon picture for the Hecke triangle groups
\[
H_q=\langle S,R\mid S^2=R^q=1\rangle\simeq C_2*C_q,
\qquad q\ge3,
\]
while keeping the geometry of the Hecke tessellation visible throughout.

The starting point is classical.  If $X\le H_q$ has finite index, the image of the finite-edge tree in $X\backslash\HH$ is a finite connected bipartite ribbon graph, denoted $\mathcal G_X$.  Equivalently, its edges are the right cosets $X\backslash H_q$; the two vertex classes are the cycles of the order-$2$ and order-$q$ generators.  After compactification this is exactly the dessin of the natural three-point quotient map.  The only local change from the prime-$q$ picture is important: a cycle of the order-$q$ generator may have any length $d\mid q$.  Thus an odd vertex of $\mathcal G_X$ may have any divisor degree $d\mid q$, and in the tessellation it is represented by a $d$-cluster with elliptic stabilizer of order $q/d$.

\begin{proposition}[Classical subgroup--graph correspondence]\label{thm:intro-graph}
For every $q\ge3$, conjugacy classes of finite-index subgroups of $H_q$ are in bijection with isomorphism classes of connected bipartite $q$-boid ribbon graphs.  If $X$ corresponds to $\mathcal G_X$, then the index $[H_q:X]$ is the number of edges of $\mathcal G_X$.  The three monodromy permutations read from the ribbon graph determine the elliptic stabilizers, cusps, cusp widths, and compactified genus.
\end{proposition}

We use this correspondence as background, not as a novelty claim.  What is less explicit in the existing literature is the finite ambiguity in passing from a fixed quotient ribbon graph to a tree diagram and then to a special polygon.  It is useful to see the three objects before introducing the general construction.  For $q=5$, Lang--Lang show that the index-$5$ subgroup $G_{5,0}(2)=G_{5,1}(2)$ has a special polygon which is one depth-one $5$-gon; in the upper half-plane its cusp vertices are
\[
0,\quad \lambda_5^{-1},\quad 1,\quad \lambda_5,\quad \infty
\]
(with $\infty$ closing the cyclic list) and its side-pairing pattern is given by their formula~(5.8) \cite[p.~407]{LL16}.  Figure~\ref{fig:intro-H5} redraws that polygon in the upper half-plane and places beside it the corresponding $5$-boid dessin and one tree diagram.

\begin{figure}[htbp]
\centering
\begin{tikzpicture}[x=2.45cm,y=2.45cm]
\def\lam{1.618034}
\def\invlam{0.618034}
\draw[->,gray!65] (-.12,0)--(1.82,0);
\draw[line width=.85pt] (0,0)--(0,1.15);
\draw[line width=.85pt] (\lam,0)--(\lam,1.15);
\draw[line width=.85pt] (0,0) arc[start angle=180,end angle=0,radius=.309017];
\draw[line width=.85pt] (\invlam,0) arc[start angle=180,end angle=0,radius=.190983];
\draw[line width=.85pt] (1,0) arc[start angle=180,end angle=0,radius=.309017];
\draw[-{Stealth[length=2.2mm]},line width=.75pt] (0,.62)--(0,.88);
\draw[-{Stealth[length=2.2mm]},line width=.75pt] (\lam,.62)--(\lam,.88);
\draw[line width=.8pt] (.309017,.275)--(.309017,.343);
\draw[line width=.8pt] (1.309017,.275)--(1.309017,.343);
\fill (.809017,.190983) circle (.025);
\node[font=\scriptsize,below] at (0,0) {$0$};
\node[font=\scriptsize,below] at (\invlam,0) {$1/\lambda_5$};
\node[font=\scriptsize,below] at (1,0) {$1$};
\node[font=\scriptsize,below] at (\lam,0) {$\lambda_5$};
\node[font=\scriptsize] at (.04,1.08) {$\infty$};
\node[font=\small] at (.81,.92) {$P_{5,0}(2)\subset\mathbb H$};
\node[align=center,font=\small] at (.81,-.28) {(a) special polygon for $G_{5,0}(2)$};

\begin{scope}[shift={(2.72,.47)},x=1cm,y=1cm]
\node[vone] (w) at (0,0) {};
\node[vzero] (b0) at (-1.22,0) {};
\node[vzero] (b1) at (.95,.70) {};
\node[vzero] (b2) at (.95,-.70) {};
\draw[edge] (w)--(b0);
\draw[edge,bend left=16] (w) to (b1);
\draw[edge,bend right=16] (w) to (b1);
\draw[edge,bend left=16] (w) to (b2);
\draw[edge,bend right=16] (w) to (b2);
\node[align=center,font=\scriptsize] at (0,-1.20) {(b) associated $5$-boid dessin};
\end{scope}

\begin{scope}[shift={(5.00,.47)},x=1cm,y=1cm]
\node[vone] (c) at (0,0) {};
\foreach \ang/\nm in {90/t0,18/t1,-54/t2,-126/t3,162/t4}{
  \node[redleaf] (\nm) at ($(c)+(\ang:.94)$) {};
  \draw[edge] (c)--(\nm);
}
\node[treeRed,font=\scriptsize] at ($(t1)+(.14,.05)$) {$a$};
\node[treeRed,font=\scriptsize] at ($(t2)+(.14,-.05)$) {$a$};
\node[treeRed,font=\scriptsize] at ($(t3)+(-.14,-.05)$) {$b$};
\node[treeRed,font=\scriptsize] at ($(t4)+(-.14,.05)$) {$b$};
\node[align=center,font=\scriptsize] at (0,-1.20) {(c) a tree diagram};
\end{scope}
\end{tikzpicture}
\caption{The same index-$5$ subgroup $G_{5,0}(2)=G_{5,1}(2)$ in three forms.  Panel (a) is a redrawing, in the upper-half-plane model, of the special polygon determined by \cite[(5.8)]{LL16}: the vertical sides are paired, the two outer semicircular even lines are paired, and the middle even line is self-paired at its even point.  Panel (b) is its quotient finite-edge dessin: one degree-$5$ odd vertex, one univalent even vertex, and two bivalent even vertices.  Cutting the two bivalent vertices gives the tree diagram in (c); equal labels mark the two pairs of red terminals and the remaining red terminal is fixed.}
\label{fig:intro-H5}
\end{figure}

Following the cutting construction used for modular tree diagrams, let $\Cuts(G)$ denote the finite set of subsets of the bivalent even vertices of a $q$-boid graph $G$ whose splitting leaves a connected tree.  Equivalently, after suppressing the bivalent even vertices, the corresponding edges are the complement of a spanning tree.  An automorphism of a $q$-boid graph always means an automorphism preserving its bipartition and ribbon structure.

\begin{theorem}[Tree diagrams over a fixed graph]\label{thm:intro-tree}
Let $G$ be a $q$-boid graph.  Every tree diagram giving back $G$ is obtained by an admissible cut system $C\in\Cuts(G)$, and
\[
\left\{\begin{array}{c}
\text{$q$-boid tree diagrams}\\[-2pt]
\text{whose associated graph is isomorphic to $G$}
\end{array}\right\}\bigg/\cong
\ \simeq\
\Aut(G)\backslash\Cuts(G).
\]
Every admissible cut system has $b_1(G)$ cuts.
\end{theorem}

This sharpens the earlier finite-to-one statement: the existence of a tree presentation comes from a spanning tree, while the theorem identifies exactly when two different choices give isomorphic tree diagrams.

There is a second ambiguity when $q$ is composite.  At an internal odd vertex of degree $2\le d<q$, the $d$ occupied branches occur consecutively in the universal $q$-star, but the abstract cyclic order does not say which of the $d$ cyclic gaps contains the $q-d$ omitted branches.  We call the choice of that gap a \emph{cluster corner}.  Let $\Corners(\mathcal T)$ be the finite set of all such choices on a tree diagram $\mathcal T$; if there is no proper internal cluster, this set has one element.

\begin{theorem}[Tree diagrams and special polygons]\label{thm:intro-polygons}
For every $q\ge3$ and every $q$-boid tree diagram $\mathcal T$, every element of $\Corners(\mathcal T)$ develops in the universal Hecke tessellation to a special polygon with its side-pairing transformations.  Conversely every such special polygon recovers the same tree diagram together with its cluster corners.  Consequently the polygon orbits lying over $[\mathcal T]$ are
\[
\Aut(\mathcal T)\backslash\Corners(\mathcal T).
\]
In particular, for prime $q$ the tree diagram determines a unique polygon orbit.
\end{theorem}

The proof is geometric.  No metric is imposed on the abstract tree: one develops its expanded finite-edge form directly into the universal Hecke tessellation.  The main points are a half-plane separation lemma for complete even lines, convexity of the developed union, uniqueness of the side-pairing maps, and a direct verification of the finite and parabolic vertex cycles in the Poincar\'e polygon theorem.

Combining Theorems~\ref{thm:intro-tree} and~\ref{thm:intro-polygons} gives the complete finite ambiguity over a fixed dessin.

\begin{corollary}[Polygon orbits over a fixed dessin]\label{cor:intro-combined}
For a $q$-boid graph $G$, let $\mathcal T_C$ be the tree obtained from $C\in\Cuts(G)$.  Then the special-polygon orbits whose associated $q$-boid graph is $G$ are
\[
\Aut(G)\backslash
\coprod_{C\in\Cuts(G)}\Corners(\mathcal T_C),
\]
where a graph automorphism transports both the cut system and the cluster corners.
\end{corollary}

A separate counting statement concerns actual normalized realizations in the fixed tessellation rather than $H_q$-orbits.  If $b$ is the number of even vertices in the expanded tree and $k=|\Corners(\mathcal T)|$, Section~5 proves that the number of normalized realizations over $[\mathcal T]$ is
\[
\frac{2bk}{|\Aut(\mathcal T)|}.
\]
Keeping this count separate from the orbit statements prevents the different finite ambiguities from being conflated.

\subsection*{Relation with earlier work}

The architecture is deliberately close to the modular theory of special polygons and tree diagrams.  In the modular case, cutting $b_1(G)$ suitable bivalent vertices of the cuboid graph produces a tree diagram, and the resulting map from tree diagrams to cuboid graphs is finite-to-one \cite[\S4.4]{Kul91}.  The present Theorem~\ref{thm:intro-tree} keeps that construction but determines the whole fibre over a fixed ribbon graph.  For Hecke groups, the special-polygon and Hecke--Farey-symbol theory, including proper $r$-clusters and independent generators, was developed first for congruence subgroups and later for general composite $q$ \cite{LLT95,LLT00,LL16}; these results provide the geometric infrastructure used here.

A second direct predecessor is Huang's generalized Hecke-polygon construction from a quotient incidence graph and a maximal tree \cite[Thm.~2.1]{Huang99}.  That gives existence of a polygon from such a choice.  Our point is narrower: for a fixed ribbon quotient we keep track of all spanning-tree choices up to automorphism and, at a proper cluster, isolate the additional complementary-gap choice required for a unique polygonal development.  Thus the paper does not claim a first graph-to-polygon construction for Hecke groups.

Finally, the subgroup--permutation--map correspondence and its interpretation through orientable maps and dessins are standard \cite{JS78,HeRead15}.  We use the finite-edge quotient because it puts those classical monodromy data in the same picture as the special-polygon geometry.  The normalizer identity
\[
\Aut(\mathcal G_X)\simeq N_{H_q}(X)/X
\]
is likewise the usual centralizer--normalizer lemma for a transitive coset action and is recorded only because the same automorphism group acts on the finite choices above.

The paper is organized accordingly.  Section~2 recalls the Hecke tessellation, clusters and special polygons, with the composite case $H_6$ used to visualize proper clusters.  Section~3 develops the $q$-boid dessin and its permutation invariants; a pair of $H_5$ examples shows why the ribbon order is essential.  Section~4 proves Theorem~\ref{thm:intro-tree}.  Section~5 proves Theorem~\ref{thm:intro-polygons}, Corollary~\ref{cor:intro-combined}, and the normalized-realization formula.  Section~6 gives the examples for $H_5,H_6,H_7$ and the two arithmetic corollaries.  Broader questions suggested by these constructions are deliberately left to separate work.

\section{Hecke groups, tessellations, and special polygons}

This section fixes the geometric language used in the rest of the paper.  The Hecke tessellation, its even, odd and finite edges, and the special-polygon/Hecke--Farey-symbol theory recalled here are standard and come from the Hecke-group literature, especially Lang--Lim--Tan and Lang--Lang \cite{LLT95,LLT00,LL16}.  No new classification theorem is claimed in this section.  Its role is to make the geometric input precise, to display proper clusters in the composite case, and to prepare the finite-edge quotient used in Section~3.

\subsection{Actions and equivalence conventions}

We use right cosets $\Omega_X=X\backslash H_q$.  If $g\in H_q$ and $\omega\in\Omega_X$, write $\omega^{\sigma_g}=\omega g$.  Products of permutations are written in right-action order:
\[
\omega^{\alpha\beta}=(\omega^\alpha)^\beta,
\qquad\text{so that}\qquad
\sigma_{gh}=\sigma_g\sigma_h.
\]
In particular, once $T=R^{-1}S$ is fixed below,
\[
\sigma_T=\sigma_R^{-1}\sigma_S.
\]

Choosing a distinguished coset edge recovers an actual subgroup as its stabilizer; after the distinguished edge is forgotten, only the conjugacy class of the subgroup remains.  We shall usually work with the latter, because the $q$-boid graph is considered up to ribbon isomorphism.  Similarly, a special polygon may be considered as an actual polygon in the fixed tessellation or up to translation by $H_q$; whenever normalized realizations are counted, this will be stated explicitly.

\subsection{The Hecke group and its tessellation}
Put
\[
\lambda_q=2\cos\frac{\pi}{q}.
\]
Let
\[
S(z)=-\frac1z,\qquad U(z)=z+\lambda_q,
\]
and put $R=SU$. Then the orientation-preserving Hecke triangle group has presentation
\[
H_q=\langle S,R\mid S^2=R^q=1\rangle\simeq C_2*C_q.
\]
We shall also use the parabolic element
\[
T=R^{-1}S=U^{-1};
\]
replacing $T$ by $T^{-1}$ does not change any cusp-width statement below.

The extended Hecke group $H_q^*$ is generated by the reflections in the three sides of the hyperbolic triangle
\[
\Delta^*=\Delta(2,q,\infty)
\]
with vertices
\[
i,\qquad \rho_q=e^{\pi i/q},\qquad \infty.
\]
The group $H_q$ is the orientation-preserving subgroup of $H_q^*$ and has index two in it. The $H_q^*$-translates of $\Delta^*$ tessellate the upper half-plane. We call the translates of $i$, $\rho_q$, and $\infty$ respectively \emph{even vertices}, \emph{odd vertices}, and \emph{free vertices} or cusps. A cusp of a finite-index subgroup $X\le H_q$ means an $X$-orbit in $H_q\cdot\infty$, equivalently a conjugacy class of maximal parabolic subgroups of $X$.

The translates of the geodesics $[i,\infty]$, $[\rho_q,\infty]$, and $[i,\rho_q]$ are called respectively even edges, odd edges, and $f$-edges. The union of all $f$-edges, with the even bivalent points retained, is denoted by $\Fq$. It is a tree: its even vertices have degree $2$ and its odd vertices have degree $q$ \cite[Sec.~2]{LL16}. This universal $f$-edge tree will be the source of all metric and development statements in Sections~4--5.

\begin{remark}
Abstractly, $\Fq$ is the Bass--Serre tree of the free product $C_2*C_q$: even and odd vertex stabilizers are conjugates of $\langle S\rangle$ and $\langle R\rangle$.  We use this only as a conceptual check; the proofs below retain the concrete tessellation and finite-edge geometry.
\end{remark}

\subsection{Clusters and special polygons}
For a divisor $d\mid q$, $1\le d<q$, a $d$-cluster is a union of $d$ consecutive special triangles around an odd vertex. Its interior angle there is $2d\pi/q$; the corresponding local elliptic stabilizer has order $q/d$. A full $q$-gon corresponds to trivial odd stabilizer. We follow \cite[Secs.~4--5]{LL16} for special polygons and Hecke--Farey symbols.

Figure~\ref{fig:H6-local-clusters} shows two actual clusters in the upper-half-plane model for $H_6$, following the notation $\Phi_d$ of \cite[\S4.1]{LL16}.  Here $\lambda_6=\sqrt3$ and the odd vertex is $e^{\pi i/6}$.  The $2$-cluster and $3$-cluster make the divisor condition visible geometrically: their odd angles are respectively $2\pi/3$ and $\pi$, so their local stabilizers have orders $3$ and $2$.

\begin{figure}[htbp]
\centering
\begin{tikzpicture}[x=3.15cm,y=3.15cm]
\def\a{0.577350}
\def\b{0.866025}
\begin{scope}[shift={(0,0)}]
\draw[->,gray!65] (-.06,0)--(1.03,0);
\draw[line width=.85pt] (0,0)--(0,1.03);
\draw[line width=.85pt] (0,0) arc[start angle=180,end angle=0,radius=.288675];
\draw[line width=.85pt] (\a,0) arc[start angle=180,end angle=120,radius=.577350];
\draw[line width=.85pt] (\b,.5)--(\b,1.03);
\fill (\b,.5) circle (.018);
\node[font=\scriptsize,right] at (\b,.52) {$e^{\pi i/6}$};
\node[font=\scriptsize,below] at (0,0) {$0$};
\node[font=\scriptsize,below] at (\a,0) {$1/\sqrt3$};
\node[font=\scriptsize] at (.025,.96) {$\infty$};
\node[font=\small] at (.43,.82) {$\Phi_2$};
\node[align=center,font=\small] at (.45,-.38) {$2$-cluster: angle $2\pi/3$\\stabilizer order $3$};
\end{scope}
\begin{scope}[shift={(1.50,0)}]
\draw[->,gray!65] (-.06,0)--(1.03,0);
\draw[line width=.85pt] (0,0)--(0,1.03);
\draw[line width=.85pt] (0,0) arc[start angle=180,end angle=0,radius=.288675];
\draw[line width=.85pt] (\a,0) arc[start angle=180,end angle=0,radius=.144338];
\draw[line width=.85pt] (\b,0)--(\b,1.03);
\fill (\b,.5) circle (.018);
\node[font=\scriptsize,right] at (\b,.52) {$e^{\pi i/6}$};
\node[font=\scriptsize,below] at (0,0) {$0$};
\node[font=\scriptsize,below] at (\a,0) {$1/\sqrt3$};
\node[font=\scriptsize,below] at (\b,0) {$\sqrt3/2$};
\node[font=\scriptsize] at (.025,.96) {$\infty$};
\node[font=\small] at (.43,.82) {$\Phi_3$};
\node[align=center,font=\small] at (.45,-.38) {$3$-cluster: angle $\pi$\\stabilizer order $2$};
\end{scope}
\end{tikzpicture}
\caption{Two proper clusters in the $H_6$ tessellation, redrawn in the upper-half-plane model from the construction of $\Phi_d$ in \cite[\S4.1]{LL16}.  For $\Phi_2$ the two odd boundary edges meet at $e^{\pi i/6}$ with angle $2\pi/3$; for $\Phi_3$ they form a straight marked angle $\pi$.  These are the geometric sources of degree-$2$ and degree-$3$ odd vertices in the quotient graph.  When such a cluster is reconstructed from an abstract tree, one must also know on which side of the occupied consecutive block the omitted sectors of the universal $6$-star lie; Section~5 records this as the cluster-corner datum.}
\label{fig:H6-local-clusters}
\end{figure}

A \emph{special polygon} $M_X=(P,I_X)$ consists of a convex finite union $P$ of full $q$-gons and $d$-clusters, normalized so that $0$ and $\infty$ are vertices, together with the allowed odd, even, and free side pairings. At a proper $d$-cluster the two boundary odd edges are paired by a conjugate of $R^d$; at an even vertex the two even half-edges are either paired by a conjugate of $S$ or together form a free side paired with another free side.

We use the following theorem of Lang--Lang, which extends the prime-$q$ special-polygon theory of Lang--Lim--Tan.

\begin{theorem}[Lang--Lang]\label{thm:LL-special}
Every finite-index subgroup $X\le H_q$ admits a special polygon $M_X=(P,I_X)$ which is an admissible fundamental domain. Its side-pairing transformations form a set of independent generators of $X$.
\end{theorem}

The present paper does not reprove this theorem. Our use of it is to compare the special polygon with the finite $f$-edge quotient and with tree diagrams obtained by cutting the finite graph.

For reference, our geometric terminology agrees with Lang--Lang as follows.
\begin{center}
\small
\begin{tabularx}{.94\textwidth}{@{}>{\raggedright\arraybackslash}p{.22\textwidth}>{\raggedright\arraybackslash}p{.25\textwidth}X@{}}
\toprule
Present paper & Lang--Lang & Meaning here\\
\midrule
special triangle & special triangle & the union $\Phi(e)$ of the two Coxeter triangles adjacent to one $f$-edge\\
$d$-cluster & $r$-cluster & $d$ consecutive special triangles at one odd centre, with $d\mid q$\\
even line / even side & even line / even side & a complete even geodesic, or one of its two rays from the even midpoint\\
free side & free side & an exposed complete even line paired with a distinct exposed even line\\
special polygon & special polygon & a finite convex union of $q$-gons and clusters with the prescribed side pairings; normalization is treated separately below\\
independent generators & independent generators & one representative from each inverse pair of side-pairing transformations\\
\bottomrule
\end{tabularx}
\end{center}

\section{Bipartite $q$-boid graphs}

This section packages the standard coset and dessin data of a finite-index Hecke subgroup in the particular bipartite ribbon graph used later.  The subgroup--permutation--dessin correspondence and the index/cusp/genus formulas are classical; Proposition~\ref{prop:dessin} and Theorem~\ref{thm:subgroup-graph} should be read as geometric reformulations adapted to the finite-edge tessellation, not as new dessin theory.  The only essential correction to the prime-$q$ picture is that an odd valency may be any divisor of $q$.  The genuinely new fibre questions begin in Section~4.

\begin{definition}\label{def:qboid}
A \emph{bipartite $q$-boid graph} is a finite connected bipartite multigraph
\[
G=(V_0\sqcup V_1,E)
\]
with the following properties:
\begin{enumerate}[label=(\roman*)]
\item $\deg(v)\in\{1,2\}$ for every $v\in V_0$;
\item $\deg(w)=d$ for some divisor $d\mid q$ for every $w\in V_1$;
\item every $w\in V_1$ of degree at least $2$ carries a prescribed cyclic order on its incident edge germs.
\end{enumerate}
Parallel edges are allowed. We write $E(G)$ for the edge set. Isomorphisms preserve the bipartition, incidences, multiplicities, and cyclic orders; hence an automorphism of a $q$-boid graph automatically preserves all of this ribbon data.
\end{definition}

For prime $q$, the condition on $V_1$ reduces to degree $1$ or $q$, exactly as in the original version of the definition. The terminology ``$q$-boid'' was suggested by Ser Peow Tan during a virtual talk in December 2021.

\begin{definition}\label{def:graph-of-subgroup}
Let $X\le H_q$ be of finite index. The associated $q$-boid graph is
\[
\mathcal G_X=X\backslash \Fq.
\]
Thus $V_0$ is the set of $X$-orbits of even vertices, $V_1$ is the set of $X$-orbits of odd vertices, and the edges are the $X$-orbits of $f$-edges. The orientation of $X\backslash\HH$ induces the cyclic order at vertices in $V_1$.
\end{definition}

There is an equivalent coset description. Let
\[
\Omega_X=X\backslash H_q
\]
be the set of right cosets, and let $\sigma_S,\sigma_R$ denote right multiplication by $S,R$. Then the graph edges are the elements of $\Omega_X$; the $V_0$-vertices are the cycles of $\sigma_S$, and the $V_1$-vertices are the cycles of $\sigma_R$. An edge labelled by $\omega\in\Omega_X$ joins the two cycles containing $\omega$.

With the right-action convention of Section~2, $\sigma_T=\sigma_R^{-1}\sigma_S$ for $T=R^{-1}S$.

For orientation, the regular level-$2$ examples already show the prime pattern that will reappear in Section~6: $H_5(2)$ has two degree-$5$ odd vertices joined through five bivalent even vertices, while $H_7(2)$ has two degree-$7$ odd vertices joined through seven bivalent even vertices.  In both cases the graph is the regular dessin of the dihedral quotient; the $H_6$ clusters in Figure~\ref{fig:H6-local-clusters} show exactly what changes when proper divisor valencies occur.

\begin{proposition}[The $q$-boid graph as a dessin]\label{prop:dessin}
Let $\overline Y_X$ be the compactified coarse Riemann surface attached to $X\backslash\HH$, and normalize the compactified quotient map
\[
\pi_X:\overline Y_X\longrightarrow \overline Y_{H_q}\simeq\mathbb P^1
\]
so that the finite orbifold branch values of orders $2$ and $q$ are $0$ and $1$, and the cusp is $\infty$.  Then
\[
\pi_X^{-1}([0,1])
\]
is the colored ribbon graph $\mathcal G_X$.  Its black vertices are the cycles of $\sigma_S$, its white vertices are the cycles of $\sigma_R$, and its faces are the cycles of $\sigma_T$.
\end{proposition}

\begin{proof}
The open arc $(0,1)$ lifts once on every sheet, hence its lifts are naturally the cosets $X\backslash H_q$.  The local monodromies about $0$ and $1$ are $\sigma_S$ and $\sigma_R$, so their cycles give the two vertex classes and the orientation of $\overline Y_X$ gives the same cyclic orders as the finite-edge quotient.  The complementary regions are the cycles of the parabolic monodromy $\sigma_T$.  Thus the dessin and $\mathcal G_X$ coincide as colored ribbon graphs.
\end{proof}

\begin{proposition}[Clusters and stabilizers]\label{prop:cluster-stab}
Let $w\in V_1(\mathcal G_X)$ have degree $d$, and choose $g\in H_q$ so that $w$ is represented by the odd vertex $g\rho_q$. Then $d\mid q$, the incident graph edges form a $d$-cycle of $\sigma_R$, and
\[
X\cap g\langle R\rangle g^{-1}=g\langle R^d\rangle g^{-1}.
\]
Thus the local stabilizer has order $q/d$. In a special polygon the corresponding tile is a $d$-cluster. When $d=q$ the local stabilizer is trivial.

At a vertex of $V_0$, degree one means an elliptic stabilizer of order $2$ and degree two means trivial even stabilizer.
\end{proposition}

\begin{proof}
Suppose the coset of $g\in H_q$ lies in an $R$-cycle of least length $d$. Then
\[
XgR^d=Xg,
\]
so $gR^dg^{-1}\in X$. Since $R$ has order $q$, minimality implies $d\mid q$ and $R^d$ has order $q/d$. The $d$ cosets in the cycle correspond to $d$ consecutive special triangles about the odd vertex, giving the $d$-cluster. The statement at $V_0$ is the identical order-two argument for $S$.
\end{proof}

\begin{proposition}[Classical subgroup--graph correspondence]\label{thm:subgroup-graph}
For every $q\ge3$, the assignment $X\mapsto\mathcal G_X$ induces a bijection
\[
\left\{\begin{array}{c}
\text{conjugacy classes of finite-index}\\[-2pt]
\text{subgroups }X\le H_q
\end{array}\right\}
\longleftrightarrow
\left\{\begin{array}{c}
\text{isomorphism classes of connected}\\[-2pt]
\text{bipartite $q$-boid graphs}
\end{array}\right\}.
\]
\end{proposition}

\begin{proof}
The forward construction is geometric. The quotient of the universal $f$-edge tree by $X$ is finite because $X$ has finite index. Proposition~\ref{prop:cluster-stab} gives the allowed degrees; the orientation of $X\backslash\HH$ gives the cyclic orders.

Conversely let $G$ be a connected bipartite $q$-boid graph with edge set $E$. Define a permutation $\sigma_S$ of $E$ by exchanging the two incident edges at every bivalent $V_0$-vertex and fixing the edge at every univalent $V_0$-vertex. Define $\sigma_R$ by cyclic successor around each $V_1$-vertex. Then
\[
\sigma_S^2=1,\qquad \sigma_R^q=1.
\]
Connectedness of $G$ is exactly transitivity of the group generated by these two permutations. Hence the presentation
\[
H_q=\langle S,R\mid S^2=R^q=1\rangle
\]
defines a transitive action of $H_q$ on $E$. The stabilizer of any chosen edge is a finite-index subgroup $X$; choosing another edge conjugates $X$. By construction, the quotient $f$-edge graph of this action is $G$.

Geometrically, the same construction glues the special triangles indexed by $E$ across the even edges according to $\sigma_S$ and around the odd vertices according to $\sigma_R$. A $d$-cycle is interpreted as an orbifold point with local stabilizer $\langle R^d\rangle$, so the missing $q-d$ sectors are accounted for by the local cone group rather than by gluing to a nonsingular point. This is precisely the quotient tessellation of the subgroup just constructed.
\end{proof}

\begin{corollary}\label{cor:index-edge}
For every finite-index subgroup $X\le H_q$,
\[
[H_q:X]=|E(\mathcal G_X)|.
\]
\end{corollary}

\begin{proposition}[Passport and geometric invariants]\label{prop:passport}
Let $n=|E(\mathcal G_X)|$ and put
\[
\sigma_T=\sigma_R^{-1}\sigma_S,
\]
corresponding to $T=R^{-1}S$. Then:
\begin{enumerate}[label=(\roman*)]
\item fixed points of $\sigma_S$ are the even elliptic points of order $2$;
\item a $d$-cycle of $\sigma_R$ with $d<q$ gives an odd elliptic point of order $q/d$;
\item cycles of $\sigma_T$ are the cusps and their lengths are the cusp widths;
\item if $g$ is the genus of the compactification of $X\backslash\HH$, then
\[
2-2g=c(\sigma_S)+c(\sigma_R)+c(\sigma_T)-n,
\]
where $c(\sigma)$ denotes the number of cycles of $\sigma$.
\end{enumerate}
\end{proposition}

\begin{proof}
The first three assertions are the local stabilizer and parabolic orbit descriptions. For the last, compactify the ribbon graph by one point for every face. Its vertices are the cycles of $\sigma_S$ and $\sigma_R$, its edges are the $n$ cosets, and its faces are the cycles of $\sigma_T$. Euler's formula gives the result.
\end{proof}

When a bivalent even vertex is suppressed, parallel edges and loops may appear in the reduced multigraph used later.  A loop contributes twice to the degree of its incident vertex, as usual.  The dessin $\mathcal G_X$ itself remains bipartite before this suppression.

\begin{corollary}[Visual torsion-free criterion]\label{cor:torsionfree}
A finite-index subgroup $X\le H_q$ is torsion-free if and only if every $V_0$-vertex of $\mathcal G_X$ has degree $2$ and every $V_1$-vertex has degree $q$.
\end{corollary}

The cyclic orders in Definition~\ref{def:qboid} are essential, not decorative.  The next two $H_5$ examples have the same underlying incidence multigraph but different face permutations.

\begin{figure}[htbp]
\centering
\begin{tikzpicture}[scale=.91]
  \node[vone] (L1) at (-4.6,0) {};
  \node[vone] (R1) at (-.9,0) {};
  \foreach \y/\i in {1.35/1,.68/2,0/3,-.68/4,-1.35/5}{
    \node[vzero] (m\i) at (-2.75,\y) {};
    \draw[edge] (L1)--(m\i)--(R1);
    \node[font=\scriptsize,above] at (m\i) {$\i$};
  }
  \node[font=\scriptsize,align=center] at (-4.6,1.88) {$V_1$ order\\$(1,2,3,4,5)$};
  \node[font=\scriptsize,align=center] at (-.9,1.88) {$V_1$ order\\$(5,4,3,2,1)$};
  \node[align=center,font=\small] at (-2.75,-2.15) {$G_5(2)$: five faces of length $2$};
  \node[vone] (L2) at (1.2,0) {};
  \node[vone] (R2) at (4.9,0) {};
  \foreach \y/\i in {1.35/1,.68/2,0/3,-.68/4,-1.35/5}{
    \node[vzero] (n\i) at (3.05,\y) {};
    \draw[edge] (L2)--(n\i)--(R2);
    \node[font=\scriptsize,above] at (n\i) {$\i$};
  }
  \node[font=\scriptsize,align=center] at (1.2,1.88) {$V_1$ order\\$(1,2,3,4,5)$};
  \node[font=\scriptsize,align=center] at (4.9,1.88) {$V_1$ order\\$(1,2,3,4,5)$};
  \node[align=center,font=\small] at (3.05,-2.15) {$G_5'$: one face of length $10$};
\end{tikzpicture}
\caption{The principal level-$2$ subgroup $G_5(2)$ and the commutator subgroup $G_5'$ have the same bipartite incidence multigraph but different ribbon orders.  The resulting parabolic cycle types are $2^5$ and $10$, respectively.  Thus the cyclic order at a $V_1$-vertex is essential to the dessin and to the cusp data.}
\label{fig:H5-ribbon-early}
\end{figure}

\begin{proposition}[Standard normalizer lemma]\label{prop:normalizer-aut}
Let $X\le H_q$ have finite index.  Then there is a canonical isomorphism
\[
\Aut(\mathcal G_X)\cong N_{H_q}(X)/X,
\]
where $\Aut$ denotes the color- and cyclic-order-preserving automorphism group.
\end{proposition}

This is the standard centralizer--normalizer lemma for a transitive coset action; compare the map construction in \cite[\S10]{LL16}.  We include the short proof because the quotient acts on the exact fibres below.

\begin{proof}
Use the right-coset model $\Omega_X=X\backslash H_q$.  A ribbon-graph automorphism commutes with the involution $\sigma_S$ at the $V_0$-vertices and with cyclic successor $\sigma_R$ at the $V_1$-vertices; hence it is exactly an automorphism of the transitive right $H_q$-set $\Omega_X$.  If $n\in N_{H_q}(X)$, left multiplication
\[
Xg\longmapsto Xng
\]
is such an automorphism.  Conversely an $H_q$-set automorphism is determined by the image $Xn$ of the base coset $X$, and equivariance forces $nXn^{-1}=X$.  The kernel consists of $n\in X$.
\end{proof}

\begin{remark}
Only color- and cyclic-order-preserving automorphisms occur here.  We exclude orientation-reversing symmetries; those belong to the corresponding normalizer problem in the extended reflection group $H_q^*$.
\end{remark}

\section{\texorpdfstring{$q$}{q}-boid tree diagrams over a fixed graph}

This section contains the first principal new result.  The operation of cutting $b_1(G)$ suitable bivalent even vertices to obtain a tree diagram is the modular construction of Kulkarni \cite[\S4.4]{Kul91}, and the use of maximal trees in Hecke-polygon constructions also appears in Huang \cite[Thm.~2.1]{Huang99}.  What is new here is the exact classification, for one fixed ribbon graph $G$, of all resulting tree diagrams up to isomorphism: the ambiguity is precisely the set of admissible cut systems modulo $\Aut(G)$.  The terminology ``cut'' is therefore inherited from the earlier construction; the finite set $\Cuts(G)$ introduced below is only notation for organizing all such choices.

\subsection{Tree diagrams and the associated graph}

\begin{definition}\label{def:tree}
A \emph{$q$-boid tree diagram} $\mathcal T$ is specified by a finite bipartite tree $\widehat{\mathcal T}$, with at least one edge, whose vertices are of exactly four types:
\begin{enumerate}[label=(\roman*)]
\item red terminal vertices, of even type;
\item blue terminal vertices, of odd type;
\item marked odd internal vertices of degree $d\mid q$, $d\ge2$, each with a cyclic order on its incident edge germs;
\item unmarked bivalent even vertices.
\end{enumerate}
Every edge joins an even vertex to an odd vertex.  The red terminals carry an involution $\iota$, possibly with fixed points.  For drawing purposes an unmarked bivalent even vertex may be suppressed, joining its two odd neighbours by one edge; conversely exactly one such vertex is reinserted on every drawn odd--odd edge.  A marked degree-two odd vertex is never suppressed.  We call $\widehat{\mathcal T}$ the \emph{expanded form} of the diagram.

An isomorphism of tree diagrams is an isomorphism of expanded forms preserving the four vertex types, the cyclic orders, and $\iota$.  The empty tree, a single vertex, and the red--red one-edge drawing are excluded.  The blue--blue one-edge drawing is allowed; its expanded form has one even midpoint and two edges.
\end{definition}

As in the modular tree diagrams, a red terminal is drawn as a hollow red circle and a blue terminal as a shaded blue circle.  Two distinct red terminals interchanged by $\iota$ carry the same label; an unlabeled red terminal is fixed by $\iota$.

\begin{definition}[Associated graph]\label{def:fold}
Given a $q$-boid tree diagram $\mathcal T$, form a bipartite ribbon graph $\mathcal G(\mathcal T)$ from its expanded form as follows.  Identify each transposed pair of red terminals to one bivalent $V_0$-vertex; retain each fixed red terminal as a univalent $V_0$-vertex and each internal even vertex as a bivalent $V_0$-vertex; and retain the marked odd vertices and blue terminals as the $V_1$-vertices.  The cyclic orders at the odd vertices are unchanged.
\end{definition}

By construction $\mathcal G(\mathcal T)$ is a connected $q$-boid graph, and the operation is independent of the drawing convention because the expanded form is unique.

\subsection{Cut systems}

Let $G$ be a $q$-boid graph.  Suppress every bivalent $V_0$-vertex and call the resulting connected multigraph $\overline G$; loops are allowed.  Following the cutting construction for modular tree diagrams, we make the spanning-tree choice explicit.

\begin{definition}\label{def:cuts}
An \emph{admissible cut system} is a subset $C$ of the bivalent $V_0$-vertices of $G$ such that the corresponding edges of $\overline G$ are precisely the complement of a spanning tree.  Equivalently, split every $c\in C$ into two red terminal vertices, one for each incident edge germ, and retain both incident edges; after suppressing the remaining bivalent $V_0$-vertices, the result is a tree.  The finite set of admissible cut systems is denoted by $\Cuts(G)$.
\end{definition}

For $C\in\Cuts(G)$, pair the two new red terminals arising from each cut and fix the original univalent $V_0$-vertices.  The resulting tree diagram is denoted $\mathcal T_C$.  Since an automorphism of $G$ preserves its bipartition and cyclic orders by definition, $\Aut(G)$ acts on $\Cuts(G)$ by
\[
\beta\cdot C=\{\beta(c):c\in C\}.
\]
It also carries $\mathcal T_C$ canonically to $\mathcal T_{\beta C}$: the red terminal arising from the incidence $(c,a)$ is sent to the one arising from $(\beta c,\beta a)$.  This description includes the case in which the corresponding edge of $\overline G$ is a loop.

\begin{lemma}\label{lem:number-cuts}
Every admissible cut system has
\[
b_1(G)=|E(\overline G)|-|V(\overline G)|+1
\]
elements.
\end{lemma}

\begin{proof}
The complement of a spanning tree of a connected finite multigraph has exactly $|E|-|V|+1$ edges.  Suppression of bivalent vertices does not change the first Betti number.
\end{proof}

\begin{theorem}[Exact graph--tree fibre]\label{thm:graph-tree-fibre}
For a fixed $q$-boid graph $G$, the isomorphism classes of $q$-boid tree diagrams whose associated graph is isomorphic to $G$ are naturally in bijection with
\[
\Aut(G)\backslash\Cuts(G).
\]
\end{theorem}

\begin{proof}
Choose a spanning tree of $\overline G$.  Splitting the complementary edges at the corresponding bivalent $V_0$-vertices produces $\mathcal T_C$, and by construction $\mathcal G(\mathcal T_C)\cong G$.

Conversely, suppose $\mathcal G(\mathcal T)\cong G$.  Every transposed red pair in $\mathcal T$ records a bivalent $V_0$-vertex at which $G$ was cut.  Since $\mathcal T$ is a tree, the uncut edges of $\overline G$ form a spanning tree; hence the recorded set belongs to $\Cuts(G)$.

If $C$ and $C'$ are carried to one another by $\beta\in\Aut(G)$, the canonical lift described above gives an isomorphism $\mathcal T_C\cong\mathcal T_{C'}$.  Conversely, an isomorphism $\mathcal T_C\cong\mathcal T_{C'}$ descends, after the paired red terminals are identified, to an automorphism of $G$ carrying $C$ to $C'$.  This proves the orbit formula.
\end{proof}

\begin{remark}
The cutting operation itself is not new: it is the one used in the modular tree-diagram construction, and an analogous maximal-tree step occurs in generalized Hecke-polygon constructions.  The content of Theorem~\ref{thm:graph-tree-fibre} is the exact description of all choices over a fixed ribbon graph, modulo its automorphisms.
\end{remark}

\section{From tree diagrams to special polygons}\label{sec:tree-polygons}

This section contains the second principal new result.  The existence of special polygons for finite-index Hecke subgroups, their side-pairing generators, and the $d$-cluster geometry are taken from Lang--Lim--Tan and Lang--Lang \cite{LLT95,LLT00,LL16}; the Poincar\'e polygon theorem used in the verification is classical.  The new point is the inverse problem for the tree diagram: at every proper cluster one must retain the complementary cyclic gap, and with precisely this extra datum every tree develops to a special polygon.  We prove the resulting exact polygon fibre and then combine it with Section~4.

\subsection{Special polygons and cluster corners}

We now make the polygon--tree construction precise for the divisor-valued tree diagrams of Section~4.  A polygon always includes its side-pairing transformations.  With $R=SU$, the point
\[
\eta_q=S\rho_q=-\overline{\rho_q}
\]
is fixed by $R$, while $SRS$ fixes $\rho_q$.  We read cyclic successor clockwise at $\eta_q$; thus the $f$-edge $e_R=[i,\eta_q]$ is adapted to the literal $(S,R)$ coset labelling.  For an oriented $f$-edge $e$, let $\Phi(e)$ be the special triangle formed by the two Coxeter triangles adjacent to $e$.

\begin{definition}[Special polygons used below]\label{def:H4-polygon}
Let $\mathscr P_q$ denote the class of special polygons before imposing the normalization $0,\infty\in P$.  An element $M=(P,I)$ consists of:
\begin{enumerate}[label=(\roman*)]
\item a nonempty finite convex union $P$ of full ideal $q$-gons and proper $d$-clusters, $d\mid q$, with disjoint tile interiors; two distinct tiles may meet in $\HH$ only along a complete even line, and otherwise only at common ideal endpoints;
\item the boundary sides determined by this tiling: the two outer odd rays of every proper cluster, and every exposed complete even line, the latter either split at its midpoint into a paired even pair or paired as a free side with another exposed even line;
\item for each side pair, the unique element of $H_q$ carrying the first side to its mate, reversing the induced boundary direction and carrying the interior side to the exterior side.
\end{enumerate}
Boundary odd centres are retained even when their angle is $\pi$.  A normalized polygon additionally has $0$ and $\infty$ as ideal vertices.  The group $H_q$ acts by
\[
h\cdot(P,I)=(hP,hIh^{-1}),
\]
and $\mathscr P_q/H_q$ denotes the corresponding orbit set.
\end{definition}

This is Lang--Lang's special-polygon class with their normalization condition omitted and with the actual side-pairing maps retained.  The finite and ideal vertex-cycle conditions required for a fundamental polygon will be verified below rather than included in the definition.

At a proper internal cluster, the cyclic order of the occupied branches does not identify the omitted sector.  This is the only additional local datum needed for the inverse construction.

\begin{definition}[Cluster corners]\label{def:corner}
Let $\mathcal T$ be a $q$-boid tree diagram and let $\widehat{\mathcal T}$ be its expanded form.  Put
\[
E_f(\mathcal T)=E(\widehat{\mathcal T}).
\]
At an internal odd vertex $v$ of degree $2\le d<q$, choose one of the $d$ cyclic gaps between consecutive incident germs; this is the \emph{cluster corner} at $v$.  Let
\[
\Corners(\mathcal T)=
\prod_{\substack{v\ {\rm internal}\\2\le d_v<q}}
\operatorname{Gap}(v),
\qquad
|\Corners(\mathcal T)|=\prod_v d_v,
\]
with the empty product equal to one.  A tree isomorphism transports a cluster corner by transporting its ordered pair of boundary germs, and hence $\Aut(\mathcal T)$ acts on $\Corners(\mathcal T)$.
\end{definition}

Figure~\ref{fig:H6-local-clusters} shows exactly what this gap records for the $2$- and $3$-clusters in $H_6$.

Given $M=(P,I)\in\mathscr P_q$, take the union of its $f$-edges, suppress only the interior bivalent even vertices, and retain the following data: a fixed red terminal at an even self-pairing, a transposed red pair at a pair of free sides, a blue terminal for a $1$-cluster, an internal odd vertex of degree $d$ for a $d$-cluster, the geometric cyclic orders, and the complementary gap at every proper internal cluster.  Denote the resulting tree with its cluster corners by $\mathcal T(M)$.

\begin{lemma}\label{lem:H2-forward}
The construction $M\mapsto\mathcal T(M)$ gives a $q$-boid tree diagram with cluster corners, is independent of incidental choices, and is $H_q$-equivariant.
\end{lemma}

\begin{proof}
The tile partition is intrinsic: two tiles in one parent $q$-gon would share its odd centre, which is excluded by Definition~\ref{def:H4-polygon}.  Convexity connects tile interiors through complete even interfaces, so their $f$-edges form a connected subgraph of the universal tree $\Fq$; hence they form a tree.  The local degrees and red involution are exactly those just listed, and a proper cluster has one complementary cyclic gap.  All of these data are preserved by the $H_q$-action.
\end{proof}

\subsection{Development and convexity}

Orient every $f$-edge from its even to its odd endpoint.  The action of $H_q$ on oriented $f$-edges is simply transitive: transitivity follows from the reflection tessellation, and an orientation-preserving isometry fixing the two finite endpoints of an $f$-edge is the identity.

\begin{definition}\label{def:development}
Let $c\in\Corners(\mathcal T)$, choose a base edge $a\in E_f(\mathcal T)$ and a target oriented $f$-edge $e_*$.  Send $a$ to $e_*$ and continue through the tree.  At an internal even vertex take the unique other universal $f$-edge; at a degree-$q$ odd vertex use the whole $q$-star; and at a proper degree-$d$ odd vertex place the $d$ incident germs consecutively so that the chosen cluster corner contains the $q-d$ omitted spokes.  Denote the resulting map by $D_{c,a,e_*}$.
\end{definition}

\begin{lemma}\label{lem:H2-development}
The development $D_{c,a,e_*}$ is unique and embeds the expanded tree in $\Fq$.  Around every odd vertex the adjacent special triangles form the prescribed full $q$-gon or $d$-cluster.
\end{lemma}

\begin{proof}
An incoming edge determines the remaining incident edges uniquely from the cyclic order and, for $d<q$, the chosen gap.  Since the source is a tree there is no compatibility cycle.  Local injectivity therefore gives global injectivity: otherwise the path between two source vertices with the same image would give a nontrivial closed walk in the universal tree.  The metric is inherited from the tessellation; no metric is placed on the abstract tree.
\end{proof}

\begin{lemma}[Separation by an even line]\label{lem:H4-separation}
Let $L$ be a complete even line.  Removing the corresponding edge from the dual tree of ideal $q$-gons gives two components, and the polygons in these components lie in the two closed half-planes bounded by $L$.
\end{lemma}

\begin{proof}
Each parent $q$-gon has connected interior disjoint from $L$, hence lies on one side of $L$.  Two parent polygons adjacent across an even line $M\ne L$ lie on the same side, as is seen in a small disc around a finite point of $M$ disjoint from $L$.  Thus the side is constant along every dual path avoiding the edge of $L$.  The two polygons incident to $L$ lie on opposite sides.  This proves the claim, and clusters inherit the same containment.
\end{proof}

\begin{lemma}[Convexity]\label{lem:H2-convexity}
For every $c\in\Corners(\mathcal T)$, the union of the developed tiles is a convex finite-area polygon.  Its boundary consists exactly of the exposed even lines at red terminals and the two outer odd rays at each odd vertex of degree $d<q$.
\end{lemma}

\begin{proof}
A proper $d$-cluster is convex because $d<q$ and $d\mid q$ imply $d\le q/2$; it is the intersection of its parent $q$-gon with a sector of angle $2\pi d/q\le\pi$.  Root the tile-adjacency tree and add tiles in parent-before-child order.  When a new tile is attached along a complete even line $L$, Lemma~\ref{lem:H4-separation} places the previous union and the new tile in opposite closed half-planes bounded by $L$, while both contain $L$.  Their union is therefore convex.  Induction proves the first assertion.  Internal even interfaces disappear from the boundary; exactly the displayed red and odd sides remain.
\end{proof}

\subsection{Side pairings and the polygon theorem}

Fix a reference special triangle $\Phi(e_0)$.  For a red terminal $r$, let $g_r e_0$ be the developed $f$-edge on the interior side of its exposed even line.  If $r$ is fixed by the red involution, pair its two even rays by
\[
s_r=g_rSg_r^{-1}.
\]
If $r$ is paired with $r'\ne r$, pair the two free sides by
\[
p_{r,r'}=g_{r'}Sg_r^{-1},
\qquad p_{r',r}=p_{r,r'}^{-1}.
\]
At an odd vertex $w$ of degree $d<q$, let $r_w$ be the primitive clockwise rotation of its universal $q$-star; the two outer odd rays are paired by
\[
t_w=r_w^d.
\]

\begin{lemma}[Uniqueness of the side maps]\label{lem:H4-side-uniqueness}
The displayed maps are the unique elements of $H_q$ with the prescribed side assignments and interior-to-exterior condition.
\end{lemma}

\begin{proof}
For even sides, the map must carry the unique even midpoint to the target midpoint and the incident interior $f$-edge to the exterior $f$-edge; simple transitivity on oriented $f$-edges forces $g_{r'}Sg_r^{-1}$.  For an odd pair, the map fixes the odd centre, and the cyclic stabilizer of that centre acts freely on the $q$ rays; hence the image ray forces $r_w^d$.
\end{proof}

We use the finite-sided orientation-preserving form of the Poincar\'e polygon theorem: a finite-area convex polygon with inverse side pairings, correct finite vertex cycles, nontrivial parabolic ideal cycles and complete quotient metric is a fundamental polygon for the discrete group generated by the side pairings.  This is the specialization of Maskit's theorem \cite[pp.~221--224]{Maskit71} used here.

\begin{lemma}[Cusp cycles]\label{lem:H4-horocycles}
Every ideal vertex cycle of the developed polygon has nontrivial parabolic cycle transformation and gives a complete cusp end.
\end{lemma}

\begin{proof}
Choose a sufficiently small $H_q$-equivariant family of disjoint cusp horoballs, starting at $\infty$ with $B_\infty(Y)=\{y>Y\}$ for $Y>1$.  Orient each boundary horocycle by transport of increasing $x$ at $\infty$.  For a cusp $c$ of the polygon, its intersection with the chosen horocycle is a compact oriented interval $J_c$ of positive length.  The side pairings carry one chosen horocycle isometrically and orientation-preservingly to the next and send the terminal endpoint of $J_c$ to the initial endpoint of the next interval in the cycle.  After transporting all intervals to one horocycle they concatenate in the positive direction, so the cycle product is a translation by
\[
L=\sum_c \operatorname{length}(J_c)>0.
\]
Thus it is a nonidentity parabolic.  The developed cusp sectors form a periodic strip, and escape to height $\infty$ has length $\int_Y^\infty dy/y=\infty$, proving completeness.
\end{proof}

\begin{proposition}[Poincar\'e verification]\label{lem:H2-poincare}
The developed polygon with the side pairings above is a special polygon and a fundamental domain for a finite-index subgroup $\Gamma\le H_q$.  Moreover
\[
\Gamma\backslash\Fq\cong\mathcal G(\mathcal T),
\qquad
[H_q:\Gamma]=|E_f(\mathcal T)|.
\]
\end{proposition}

\begin{proof}
Lemmas~\ref{lem:H2-development}--\ref{lem:H2-convexity} give a finite-area convex polygon with finitely many sides.  Lemma~\ref{lem:H4-side-uniqueness} gives inverse side pairings that reverse the induced boundary directions and map the interior collar to the exterior collar.

The finite vertex cycles are exactly the expected ones.  A fixed red terminal has angle $\pi$ and cycle transformation of order $2$.  A degree-$d<q$ odd vertex has angle $2\pi d/q$ and cycle transformation of order $q/d$.  Paired free-side midpoints, after the harmless auxiliary subdivision, have total angle $2\pi$ and identity cycle transformation.  These exhaust the finite boundary cycles.  Lemma~\ref{lem:H4-horocycles} gives the nontrivial parabolic ideal cycles and complete cusp ends; after the cusp caps are removed the remaining quotient is compact.  The Poincar\'e polygon theorem therefore applies and gives discreteness, the fundamental-domain property and no extra identifications.

The polygon contains one special triangle for each edge of the expanded tree, so the covolume ratio gives $[H_q:\Gamma]=|E_f(\mathcal T)|$.  In the quotient finite-edge graph, a transposed red pair identifies exactly the corresponding two even endpoints, a fixed red terminal gives the even order-$2$ stabilizer, and a degree-$d$ odd star closes with stabilizer order $q/d$.  Hence the quotient graph is precisely $\mathcal G(\mathcal T)$.
\end{proof}

\begin{theorem}[Tree--polygon reconstruction]\label{thm:enhanced-reconstruction}
For every $q\ge3$ and every $q$-boid tree diagram $\mathcal T$, every element $c\in\Corners(\mathcal T)$ is admissible.  Developing $(\mathcal T,c)$ from a base edge produces a special polygon, and taking the finite-edge tree of a special polygon is inverse to this construction up to $H_q$-translation and tree isomorphism.  Thus
\[
\mathscr P_q/H_q
\longleftrightarrow
\left\{\text{$q$-boid tree diagrams with chosen cluster corners}\right\}/\cong
\]
is a bijection.
\end{theorem}

\begin{proof}
Existence follows from the preceding lemmas.  Starting with a developed tree and then taking its finite-edge tree recovers the same odd stars, red involution and complementary gaps.  Conversely, starting with a polygon, taking $\mathcal T(M)$ and developing one edge back to its original position recovers the same finite-edge tree, hence the same intrinsic tile partition; Lemma~\ref{lem:H4-side-uniqueness} then recovers the same side pairings.  Changing the target oriented $f$-edge changes the polygon by the unique element of $H_q$ carrying one target edge to the other.
\end{proof}

\begin{corollary}[Polygon choices over a tree diagram]\label{cor:H2-corner-fibre}
For a fixed $q$-boid tree diagram $\mathcal T$, the polygon orbits that forget to $[\mathcal T]$ are exactly
\[
\Aut(\mathcal T)\backslash\Corners(\mathcal T).
\]
For prime $q$ this set has one element.
\end{corollary}

\begin{proof}
Every choice of cluster corners is admissible, and two choices yield the same polygon orbit exactly when an automorphism of the underlying tree transports one choice to the other.
\end{proof}

\subsection{Normalization and the combined fibre}

For a fixed choice of cluster corners, changing the target oriented $f$-edge translates the development by a unique element of $H_q$.  Reflection would reverse every cyclic order and every chosen gap and belongs to the extended group $H_q^*$; it is not part of the equivalence used here.

To count actual normalized realizations, let
\[
e_+=[i,\rho_q],\qquad e_-=[i,\eta_q].
\]
For a tree $\mathcal T$, set
\[
B(\mathcal T)=\Corners(\mathcal T)\times E_f(\mathcal T)\times\{+,-\}.
\]
If $a,\bar a$ meet at an internal even vertex, impose
\[
(c,a,+)\sim(c,\bar a,-),\qquad
(c,a,-)\sim(c,\bar a,+),
\]
and write $A(\mathcal T)=B(\mathcal T)/\sim$.  The group $\Aut(\mathcal T)$ acts by
\[
\alpha(c,a,\varepsilon)=(\alpha_*c,\alpha a,\varepsilon).
\]

\begin{proposition}[Normalized realizations]\label{prop:H2-normalized}
Let $b=|V_0(\widehat{\mathcal T})|$ and $k=|\Corners(\mathcal T)|$.  The normalized special polygons in the fixed tessellation whose underlying tree is isomorphic to $\mathcal T$ are
\[
\Aut(\mathcal T)\backslash A(\mathcal T),
\]
and their number is
\[
\frac{2bk}{|\Aut(\mathcal T)|}.
\]
For a fixed corner orbit represented by $c$, the number is
\[
\frac{2b}{|\Aut(\mathcal T)_c|},
\qquad
\Aut(\mathcal T)_c=\{\alpha\in\Aut(\mathcal T):\alpha_*c=c\}.
\]
\end{proposition}

\begin{proof}
A triple $(c,a,\varepsilon)$ means that the base edge $a$ is developed to $e_\varepsilon$.  At an internal even vertex, using the other incident edge with the opposite sign gives the same normalized development; these are exactly the displayed relations.  Conversely, a normalized convex polygon contains the geodesic $(0,\infty)$ and hence $i$, so it contains one of the two special triangles incident to $e_+$ or $e_-$; pulling that edge back gives a representative triple.  Two triples give the same unlabelled polygon precisely when an automorphism of the tree transports the cluster corners and base edge.

For each choice of cluster corners every even vertex contributes two equivalence classes, so $|A(\mathcal T)|=2bk$.  The automorphism action on $A(\mathcal T)$ is free: fixing a class fixes an expanded edge, hence its adjacent cyclic star, and induction along the tree fixes the whole diagram.  Division by the automorphism group gives the formula and its stabilizer refinement.
\end{proof}

Now let $G$ be a fixed $q$-boid graph and put
\[
\mathcal E(G)=
\coprod_{C\in\Cuts(G)}\Corners(\mathcal T_C).
\]
For $\beta\in\Aut(G)$, the canonical lift $\mathcal T_C\to\mathcal T_{\beta C}$ transports the chosen gaps; thus
\[
\beta\cdot(C,c)=(\beta C,\beta_*c)
\]
defines an action on $\mathcal E(G)$.

\begin{proposition}[Combined graph--polygon fibre]\label{prop:H2-naturality}
For a fixed $q$-boid graph $G$, the special-polygon orbits whose associated graph is $G$ are
\[
\Aut(G)\backslash
\coprod_{C\in\Cuts(G)}\Corners(\mathcal T_C).
\]
These correspondences preserve the even $S$-type, odd $R$-type, and parabolic cusp type separately.
\end{proposition}

\begin{proof}
Theorem~\ref{thm:graph-tree-fibre} classifies the tree diagrams over $G$.  Theorem~\ref{thm:enhanced-reconstruction} and Corollary~\ref{cor:H2-corner-fibre} then classify the polygon choices over each tree.  The development is natural under tree isomorphisms because both the cyclic orders and chosen gaps are transported.  Colors and odd degrees are preserved, so even stabilizers, odd stabilizers and cusp cycles cannot be interchanged.
\end{proof}

\begin{corollary}[Normalizer form]\label{cor:normalizer-fibre}
If $G=\mathcal G_X$, then Proposition~\ref{prop:normalizer-aut} identifies the preceding automorphism action with the natural action of $N_{H_q}(X)/X$, and hence
\[
\left\{[M]\in\mathscr P_q/H_q:\mathcal G(\mathcal T(M))\cong G\right\}
\cong
\bigl(N_{H_q}(X)/X\bigr)\backslash
\coprod_{C\in\Cuts(G)}\Corners(\mathcal T_C).
\]
\end{corollary}

\section{Examples and arithmetic corollaries}\label{sec:applications}

This final section tests the preceding constructions on explicit Hecke subgroups.  The index, elliptic, cusp and genus data for the principal examples are drawn from the established Hecke--Farey and congruence literature \cite{LL16,LLT00}; our contribution here is mainly to display them uniformly as dessins and tree diagrams.  Proposition~\ref{prop:mcmullen-faces} is an interpretation of McMullen's level-$2$ cusp invariant in this dessin language, and Proposition~\ref{prop:SS-dessins} is a Belyi-pair corollary of Schmidt--Smith's curve-level Galois theorem.  Neither is used in the proof of the main graph--tree--polygon results.  We use the passport convention
\[
(\sigma_S,\sigma_R,\sigma_T),\qquad T=R^{-1}S,
\]
with the right-action composition fixed in Section~2.  Exponential partition notation $m^a$ means $a$ disjoint cycles of length $m$.  We write
\[
D_q=\langle r,t\mid r^q=t^2=1,\ trt=r^{-1}\rangle,
\qquad |D_q|=2q.
\]
Our notation is $H_q$ throughout; when quoting sources that use $G_q$, the latter denotes the same Hecke group.

\subsection{Symmetry and level two}

Proposition~\ref{prop:normalizer-aut} makes the symmetry of a regular example transparent.  If $X\triangleleft H_q$, then
\[
\Aut(\mathcal G_X)\cong H_q/X.
\]
Thus the regular level-$2$ examples below have dihedral ribbon automorphism groups, while the principal examples with quotients $A_5$ and $\PSL_2(\mathbf F_7)$ carry those groups as orientation-preserving dessin automorphisms.  By contrast, for odd prime $q$ the index-$q$ subgroup $G_{q,0}(2)$ corresponds in the dihedral quotient to a reflection stabilizer; that stabilizer is self-normalizing, so its $q$-boid graph has no nontrivial rotation-preserving automorphism.

For the prime examples $q=5,7$, the level-$2$ quotients and the corresponding special polygons are described in \cite{LL16,LLT00}.  More generally, for every odd $q$ the explicit conjugacy with McMullen's model below gives
\[
H_q/H_q(2)\cong D_q,
\qquad
[H_q:H_q(2)]=2q,
\]
and hence the regular passport
\[
(2^q,q^2,2^q).
\]
Thus $H_q(2)$ is torsion-free of genus zero with $q$ cusps, all of width $2$.

The same quotient is the starting point of McMullen's recent work \cite{McMullenCusps}.  Put
\[
\epsilon_q=2\cos(2\pi/q),\qquad
\lambda_q=2\cos(\pi/q),\qquad \lambda_q^2=2+\epsilon_q.
\]
For odd $q$, McMullen's group $\Delta_q\subset\SL_2(\mathcal O_K)$ has
\[
\mathcal O_K=\mathbb Z[\epsilon_q]=\mathbb Z[\lambda_q],
\]
and $\lambda_q$ is a unit.  Write
\[
T_M=\begin{pmatrix}1&\lambda_q^2\\0&1\end{pmatrix},
\qquad
S_M=\begin{pmatrix}0&\lambda_q\\-\lambda_q^{-1}&0\end{pmatrix}.
\]
The second matrix is McMullen's order-two element in the odd case.  If
\[
C_q=\begin{pmatrix}\sqrt{\lambda_q}&0\\0&1/\sqrt{\lambda_q}\end{pmatrix},
\qquad
S_H=\begin{pmatrix}0&-1\\1&0\end{pmatrix},
\qquad
U_H=\begin{pmatrix}1&\lambda_q\\0&1\end{pmatrix},
\]
then
\[
C_qS_HC_q^{-1}=-S_M,
\qquad
C_qU_HC_q^{-1}=T_M.
\]
Thus $C_q$ gives an explicit projective conjugacy $H_q\simeq\Delta_q/(\pm I)$.  Conjugation scales the two off-diagonal entries by the units $\lambda_q$ and $\lambda_q^{-1}$, so it preserves the principal congruence condition modulo $2$.  Since $I=-I$ in characteristic two, passage to the projective group does not enlarge the kernel.  Hence the projective principal level-$2$ kernels correspond.

McMullen proves that the reduction of $\Delta_q$ modulo $2$ is the dihedral group $D_q$ and defines
\[
\delta:\Delta_q\cdot\infty\longrightarrow D_q/\langle t\rangle\cong\mathbb Z/q
\]
by the quotient map \cite[Thm.~1.1 and Cor.~1.3]{McMullenCusps}.

\begin{proposition}[McMullen's invariant and level-$2$ faces]\label{prop:mcmullen-faces}
For every odd $q\ge3$, under the explicit conjugacy above, two cusps have the same value of McMullen's invariant $\delta$ if and only if they determine the same cusp of $H_q(2)$.  Equivalently, the face set of the regular level-$2$ $q$-boid dessin is canonically the $D_q$-set
\[
D_q/\langle t\rangle.
\]
After choosing McMullen's base cusp and his identification of this set with $\mathbb Z/q$, $\delta$ is a numerical labelling of those faces.
\end{proposition}

\begin{proof}
McMullen identifies the cusp orbit with $\Delta_q/\langle T_M,-I\rangle$ and $\delta$ with the quotient to $D_q/\langle t\rangle$.  The conjugacy above carries the standard Hecke cusp subgroup $\langle U_H\rangle=\langle T^{-1}\rangle$ to $\langle T_M\rangle$ and carries $H_q(2)$ to the projective level-$2$ kernel.  Since the kernel is normal,
\[
H_q(2)\backslash H_q/\langle T\rangle
\cong D_q/\langle t\rangle.
\]
The double cosets on the left are the cusp classes of $H_q(2)$, equivalently the face cycles of its dessin by Proposition~\ref{prop:dessin}.  This proves the claim.
\end{proof}

For $q=5$ and $7$ this gives a modern arithmetic interpretation of the passports
\[
H_5(2):\ (2^5,5^2,2^5),
\qquad
H_7(2):\ (2^7,7^2,2^7).
\]
The five and seven width-$2$ faces are precisely the fibres of the corresponding dihedral cusp label.

\subsection{The group $H_5$: ribbon order and an $A_5$ quotient}

The subgroup $G_{5,0}(2)=G_{5,1}(2)$ has index $5$, genus $0$, one even elliptic point of order $2$, and three cusps of widths $1,2,2$.  Its passport is
\[
(1\,2^2,\ 5,\ 1\,2^2).
\]
Thus its $5$-boid graph consists of one degree-$5$ $V_1$-vertex, one univalent $V_0$-vertex, and two bivalent $V_0$-vertices.

Figure~\ref{fig:intro-H5} already displays this subgroup as a special polygon, dessin, and tree diagram.

A second example shows why the ribbon structure, rather than the underlying incidence multigraph alone, is essential.  The principal subgroup $G_5(2)$ and the commutator subgroup $G_5'$ both have index $10$ and the same underlying multigraph: two degree-$5$ $V_1$-vertices joined through five bivalent $V_0$-vertices.  Their cyclic orders differ.  The passports are
\[
G_5(2):\ (2^5,5^2,2^5),
\qquad
G_5':\ (2^5,5^2,10),
\]
so the first has five width-$2$ cusps and genus $0$, whereas the second has one width-$10$ cusp and genus $2$.

\begin{proposition}\label{prop:H5-ribbon}
The cyclic-order data in a $q$-boid graph are essential: $G_5(2)$ and $G_5'$ have the same underlying bipartite incidence multigraph but nonisomorphic ribbon structures, and their parabolic cycle types are respectively $2^5$ and $10$.
\end{proposition}

\begin{proof}
Label the five bivalent vertices by $1,\ldots,5$.  Fix the cyclic order $(1,2,3,4,5)$ at one $V_1$-vertex.  The reverse order at the other vertex gives five $2$-cycles for the parabolic permutation, while the same order gives one $10$-cycle.  Proposition~\ref{prop:passport} gives the cusp and genus data.
\end{proof}

This distinction is shown geometrically in Figure~\ref{fig:H5-ribbon-early}.

The principal congruence subgroup $G_5(3)$ satisfies
\[
G_5/G_5(3)\cong A_5,\qquad [G_5:G_5(3)]=60.
\]
It is torsion-free, has genus $0$, and has $20$ cusps, all of width $3$ \cite{LLT00}.  Thus
\[
(\sigma_S,\sigma_R,\sigma_T)=(2^{30},5^{12},3^{20}).
\]
Its $5$-boid graph has $60$ edges, $30$ bivalent $V_0$-vertices, $12$ degree-$5$ $V_1$-vertices, and $20$ faces of length $3$.  Proposition~\ref{prop:normalizer-aut} identifies its ribbon automorphism group with $A_5$.

\subsection{The group $H_7$}

The level-$2$ subgroup has passport $(2^7,7^2,2^7)$, genus $0$, and seven width-$2$ cusps.  The commutator subgroup $G_7'$ has passport $(2^7,7^2,14)$, hence one width-$14$ cusp and genus $3$ \cite{LL16}.  The ramified prime-level subgroup $G_7(\mathfrak p_7)$, where $\mathfrak p_7\mid(7)$ has residue field $\mathbf F_7$, satisfies
\[
G_7/G_7(\mathfrak p_7)\cong\PSL_2(\mathbf F_7),
\qquad [G_7:G_7(\mathfrak p_7)]=168,
\]
is torsion-free, and has passport
\[
(2^{84},7^{24},7^{24}),
\]
so it has $24$ width-$7$ cusps and genus $19$ \cite{LLT00}.  Again its ribbon automorphism group is the congruence quotient.

\begin{table}[ht]
\centering
\caption{Two regular prime-level examples.}
\label{tab:prime-level}
\begin{tabular}{@{}lcc@{}}
\toprule
& $G_5(3)$ & $G_7(\mathfrak p_7)$\\
\midrule
finite quotient & $A_5$ & $\PSL_2(\mathbf F_7)$\\
index & $60$ & $168$\\
passport & $(2^{30},5^{12},3^{20})$ & $(2^{84},7^{24},7^{24})$\\
cusps & $20$ & $24$\\
cusp widths & $3$ & $7$\\
genus & $0$ & $19$\\
\bottomrule
\end{tabular}
\end{table}

\subsection{The composite case $H_6$}

Here $\lambda_6=\sqrt3$, and proper degrees $2$ and $3$ occur on the $V_1$-side.  Let
\[
M_3=\ker\chi,\qquad
\chi:H_6\longrightarrow C_2,\qquad
\chi(S)=\chi(R)=1.
\]
Then $M_3$ has index $2$ and passport
\[
(2,2,1^2).
\]
Its degree-$2$ odd vertex has stabilizer of order $3$.  It is important that $M_3$ is \emph{not} the characteristic square subgroup.  If
\[
H_6^2=\langle g^2:g\in H_6\rangle,
\]
then
\[
H_6/H_6^2\cong C_2\times C_2,
\qquad [H_6:H_6^2]=4,
\]
and the regular four-sheet action has passport $(2^2,2^2,2^2)$.  Thus the index-two example should be viewed as one of the three character kernels of the mod-two abelianization, while the square subgroup is their intersection.

Two index-$3$ examples based on a $3$-cluster are also useful.  The normal subgroup $X$ of \cite[Sec.~5.1(i)]{LL16} has passport $(1^3,3,3)$, while $G_{6,0}(2)$ has passport $(1\,2,3,1\,2)$.  Both have an odd order-$2$ point coming from a degree-$3$ vertex, but their even elliptic and cusp data differ.

\begin{figure}[ht]
\centering
\begin{tikzpicture}[scale=1]
  \node[vone] (a) at (-4.6,0) {};
  \node[redleaf] (a1) at (-5.8,0) {};
  \node[redleaf] (a2) at (-3.4,0) {};
  \draw[edge] (a1)--(a)--(a2);
  \node[treeRed,above] at (a1) {$1$};
  \node[treeRed,above] at (a2) {$1$};
  \draw[treeRed,line width=.9pt] ($(a)+(0.18,.18)$) arc[start angle=45,end angle=135,radius=.25];
  \node[align=center] at (-4.6,-1.25) {$M_3=\ker\chi$: degree $2$\\cluster};
  \node[vone] (b) at (0,0) {};
  \foreach \ang/\nm in {90/b1,210/b2,330/b3}{
    \node[redleaf] (\nm) at ($(b)+(\ang:1.15)$) {};
    \draw[edge] (b)--(\nm);
  }
  \draw[treeRed,line width=.9pt] ($(b)+(0.18,.18)$) arc[start angle=35,end angle=125,radius=.22];
  \node[align=center] at (0,-1.65) {$X$: three fixed\\red leaves};
  \node[vone] (c) at (4.5,0) {};
  \foreach \ang/\nm in {90/c1,210/c2,330/c3}{
    \node[redleaf] (\nm) at ($(c)+(\ang:1.15)$) {};
    \draw[edge] (c)--(\nm);
  }
  \node[treeRed,left] at (c2) {$1$};
  \node[treeRed,right] at (c3) {$1$};
  \draw[treeRed,line width=.9pt] ($(c)+(0.18,.18)$) arc[start angle=35,end angle=125,radius=.22];
  \node[align=center] at (4.5,-1.65) {$G_{6,0}(2)$: one fixed\\leaf and one paired pair};
\end{tikzpicture}
\caption{Proper-cluster tree diagrams in $H_6$.  The short red arc records a cluster corner.}
\label{fig:H6-clusters}
\end{figure}

\begin{proposition}\label{prop:H6-essential}
For composite $q$ the local condition $\deg(w)\in\{1,q\}$ is insufficient.  The examples above in $H_6$ realize degree-$2$ and degree-$3$ vertices of type $V_1$, with stabilizers of orders $3$ and $2$ respectively.  Hence the divisor condition $\deg(w)\mid q$ is essential.
\end{proposition}

\begin{remark}
The mod-two abelianization calculation above also corrects the common shorthand that identifies the displayed index-two group with the subgroup generated by squares.  The divisor-valency conclusion is unchanged.
\end{remark}

When $q$ is even the coloring also matters: an order-$2$ elliptic point may be of even $S$-type or of odd $R^{q/2}$-type, and the $q$-boid graph distinguishes these two conjugacy types.

\subsection{Galois orbits of principal congruence Hecke dessins}\label{subsec:SS}

Schmidt and Smith prove that, under explicit hypotheses, principal congruence Hecke curves attached to prime ideals above one rational prime form a complete Galois orbit \cite{SchmidtSmith03}.  Their argument contains enough additional structure to retain the natural quotient map.

\begin{proposition}[A Belyi-pair consequence of Schmidt--Smith]\label{prop:SS-dessins}
Assume the hypotheses of the main theorem of \cite{SchmidtSmith03}: let $p$ be an odd prime, let $n$ be odd with $(n,p)=1$, put $m=np^a$, and suppose $p,n>5$ and the relevant residue degree $f$ is odd.  For a prime ideal $P_i\mid(p)$ let $X_i$ be the compactification of $G_m(P_i)\backslash\mathbb H$ and let
\[
\beta_i:X_i\longrightarrow \overline{G_m\backslash\mathbb H}\cong\mathbb P^1
\]
be the compactified quotient map, normalized so that the three branch points of ramification orders $2,n,p$ are $0,1,\infty$ in that order.  Then the normalized Belyi pairs $(X_i,\beta_i)$ form the same complete Galois orbit as the curves $X_i$.

If $Q=p^f$ and $d=|\PSL_2(\mathbf F_Q)|$, the regular dessin has passport
\[
\bigl(2^{d/2},\ n^{d/n},\ p^{d/p}\bigr).
\]
Viewed as an $m$-boid graph, every $V_1$-vertex has degree $n\mid m$ and stabilizer order $m/n=p^a$.
\end{proposition}

\begin{proof}
Schmidt--Smith prove that the natural deck group and the full automorphism group of the smooth compactification have the same order and are both isomorphic to $\PSL_2(\mathbf F_Q)$ \cite[Props.~2 and~5]{SchmidtSmith03}:
\[
G_m/G_m(P_i)\simeq \Aut(X_i)\simeq \PSL_2(\mathbf F_Q).
\]
Since the deck action is faithful, it is the full automorphism group.  Consequently the compactified Hecke quotient is intrinsically the quotient morphism
\[
\beta_i:X_i\longrightarrow X_i/\Aut(X_i)\simeq\mathbb P^1.
\]
Schmidt--Smith compute its three inertia orders as $2,n,p$ \cite[Prop.~3]{SchmidtSmith03}.  These are pairwise distinct under their hypotheses, so there is a unique coordinate on the target carrying the corresponding branch points to $0,1,\infty$ in that order.

For $\sigma\in\Gal(\overline{\mathbb Q}/\mathbb Q)$, conjugation transports $\Aut(X_i)$ functorially to $\Aut(X_i^\sigma)$; therefore it transports the quotient morphism to the quotient morphism of the conjugate curve.  Because the three inertia orders distinguish the normalized target points, the Galois stabilizer of the normalized pair $(X_i,\beta_i)$ is the same as the stabilizer of $X_i$.  Schmidt--Smith's complete orbit of curves therefore gives the same complete orbit of normalized Belyi pairs.  This is a corollary of their theorem, not a statement they formulate explicitly.

Finally the deck group acts regularly on $d$ sheets.  Elements of orders $2,n,p$ have cycle types $2^{d/2},n^{d/n},p^{d/p}$.  The $R$-cycle length is $n$, so Proposition~\ref{prop:cluster-stab} gives stabilizer order $m/n=p^a$.
\end{proof}

When $a>0$, Proposition~\ref{prop:SS-dessins} supplies an infinite arithmetic family of proper divisor-valency $m$-boid dessins.  This complements the small $H_6$ pictures: the latter make the local geometry visible, while the Schmidt--Smith family shows that the divisor phenomenon persists systematically in principal congruence arithmetic.

\medskip
The examples in this section serve only to illustrate the finite structures developed above.  Questions about cusp sets for two-cusp triangle groups, arithmetic congruence invariants beyond level two, Galois actions for more general triangle groups, and computational implementations are natural continuations, but they are not needed for the results of the present paper and are left to separate work.

\end{document}